\documentclass[11pt,reqno]{article}

\usepackage{enumitem}
\usepackage{setspace}
\usepackage{amsmath}
\usepackage{amsfonts}
\usepackage{amssymb}
\usepackage{wasysym}
\usepackage{amsthm}
\usepackage{mathrsfs}
\usepackage{amscd}
\usepackage{mathtools}
\usepackage{tabu}
\usepackage{tikz}
\usepackage{tikz-cd}
\usepackage{hyperref}
\usepackage[margin=1in]{geometry}
\usepackage[matrix,arrow,curve,frame]{xy}
\usepackage{xy}
\usepackage[capitalise,noabbrev]{cleveref}
\usepackage{textcomp}
\usepackage{breqn}

\definecolor{my-linkcolor}{rgb}{0.75,0,0}
\definecolor{my-citecolor}{rgb}{0.1,0.57,0}
\definecolor{my-urlcolor}{rgb}{0,0,0.75}
\hypersetup{
	colorlinks, 
	linkcolor={my-linkcolor},
	citecolor={my-citecolor}, 
	urlcolor={my-urlcolor}
}

\title{Coset Vertex Algebras}
\author{Thomas Creutzig, Davide Gaiotto and Andrew R. Linshaw}
\date{}

\numberwithin{equation}{section}

\theoremstyle{definition}\newtheorem{rema}{Remark}[section]
\theoremstyle{plain}
\newtheorem{theo}[rema]{Theorem}
\theoremstyle{definition}
\theoremstyle{plain}\newtheorem{lemma}[rema]{Lemma}
\newtheorem{corol}[rema]{Corollary}

\theoremstyle{definition}
\theoremstyle{definition}\newtheorem{ques}[rema]{Question}
\theoremstyle{definition}
    
\theoremstyle{definition}

\newcommand{\cI}{\mathcal{I}}

\newcommand{\cC}{\mathcal{C}}

\newcommand{\ZZ}{\mathbb{Z}}

\DeclareMathOperator{\ch}{ch}

\newcommand{\VOA}{{vertex operator algebra}}
\newcommand{\VOSA}{{vertex operator super-algebra}}
\newcommand{\VOAs}{{vertex operator algebras}}
\newcommand{\VOSAs}{{vertex operator super-algebras}}

\newcommand{\largeVir}[2]{\text{sVir}_{lrg\ N=4}^{\left(#1, #2\right)}}
\newcommand{\smallVir}[1]{\text{sVir}_{sm\ N=4}^{\left(#1\right)}}

\usetikzlibrary{knots}

\usepackage{xparse}
\NewDocumentCommand\render{sg}{%
	\IfBooleanTF#1%
	{#2}  
	{}    
}

\newcommand{\fg}{\mathfrak{g}}

\newcommand{\fsu}{\mathfrak{su}}
\newcommand{\fsl}{\mathfrak{sl}}
\newcommand{\fosp}{\mathfrak{osp}}

\newcommand{\bZ}{\mathbb{Z}}

\newcommand{\Vir}{\mathrm{Vir}}
\newcommand{\sVir}{\mathrm{sVir}}

\allowdisplaybreaks

\begin{document}
	
\title{S-duality for the large $N=4$ superconformal algebra}	
\date{}
\maketitle
\abstract{We prove some conjectures about vertex algebras which emerge in gauge theory constructions associated to the geometric Langlands program. In particular, we present the conjectural kernel vertex algebra for the $S T^2 S$ duality transformation in $SU(2)$ gauge theory.
We find a surprising coincidence, which gives a powerful hint about the nature of the corresponding duality wall. 

Concretely, we determine the branching rules for the small $N=4$ superconformal algebra at central charge $-9$ as well as for the generic large $N=4$ superconformal algebra at central charge $-6$. Moreover we obtain the affine vertex superalgebra of $\mathfrak{osp}(1|2)$ and the $N=1$ superconformal algebra times a free fermion as Quantum Hamiltonian reductions of the  large $N=4$ superconformal algebras at $c=-6$. 
}

\setcounter{tocdepth}{2}
\setcounter{secnumdepth}{4}


\allowdisplaybreaks

\section{Introduction}

Certain deformable families of vertex algebras and their module categories appear in gauge theory constructions \cite{CG, GR} involving the GL twist of 
four-dimensional gauge theory, with applications to problems in the quantum geometric Langlands program \cite{KWi, Gai1, Gai2}. 
The objective of this work is to prove some of the vertex algebra predictions associated to dualities in $SU(2)$ gauge theory.  

\subsection{VOA predictions from gauge theory}\label{sec:pred}

The vertex algebra constructions which follow from gauge theory can be described without reference to the original 
gauge theory motivations. We will do so here and postpone the gauge theory interpretation to a later subsection. 

The predictions are most clear for simply-laced Lie algebras. So let $\fg$ be a simply-laced Lie algebra and $\kappa$ a non-rational complex number. 
Our basic building blocks will be vertex algebras associated to $\fg$ at (critically shifted) levels which are related to $\kappa$ by
$PSL_2(\bZ)$ fractional linear transformations. These include in particular the universal affine vertex algebra $\fg_{\kappa'} \equiv V_{k'}(\fg)$ 
of $\fg$ at level $k'=\kappa'-h^\vee$. Here $h^\vee$ denotes the dual Coxeter number of $\fg$. The other building blocks are ${\mathcal W}^\rho_{\kappa'}(\fg) \equiv W_{k'}^\rho(\fg)$, the quantum Hamiltonian reduction of  $\fg_{\kappa'}$ for the embedding $\rho: \mathfrak{sl}_2 \rightarrow \fg$ of $\mathfrak{sl}_2$ in $\fg$. 
Sometimes, auxiliary rational vertex algebras also appear, such as the algebra $F(n)$ of $n$ real free fermions or the WZW vertex algebras $L_k(\fg)$, 
i.e. the simple quotients of Kac-Moody algebras at integral level. 

An important set of predictions is the existence of a collection of ``kernel vertex algebras'' $A^n[\fg, \kappa]$ which are built 
as extensions of $\fg_\kappa \times \fg_{\frac{\kappa}{n \kappa -1}}$ in the following manner.\footnote{More intricate vertex algebras $A^{(n_i)}[\fg, \kappa]$ can be built as extensions of longer chains of vertex algebras of the form 
$\fg_\kappa \otimes \left( \bigotimes\limits_{i=1}^m {\mathcal W}_{\kappa_i}(\fg)\right) \otimes \fg_{\kappa_{m+1}}$, associated to a chain of levels satisfying
$\kappa_i + \kappa_{i+1}^{-1} =n_i$
with $\kappa_0 = \kappa^{-1}$. Here ${\mathcal W}_{\kappa}(\fg)$ denotes ${\mathcal W}^\rho_{\kappa'}(\fg)$ for regular embedding $\rho$. 
We will not discuss this larger family of vertex algebras here. 
}
Let $P^+$ be the set of dominant weights of $\fg$ and $Q$ its root lattice. For $\lambda \in P^+$ denote by $V_k(\lambda)$ the Weyl module of $V_k(\fg)$ at level $k$. Let $k, k'$ be related to $\kappa, \kappa'$ by $\kappa=k+h^\vee$ and $\kappa'=k'+h^\vee$. 
Then \cite{CG} conjectures 
that for $n$ a positive integer and $\kappa, \kappa'$ such that
\[
\frac{1}{\kappa} + \frac{1}{\kappa'} =n
\]
the module
\[
A^n[\fg, \kappa] = \bigoplus_{\lambda \in P^+\cap Q} V_k(\lambda) \otimes V_{k'}(\lambda)
\]
can be given the structure of a simple vertex algebra for generic $\kappa$. 

There are several more families of vertex algebra with a related conjectural construction:
\begin{itemize}
\item One can now consider an embedding $\rho: \mathfrak{sl}_2 \rightarrow \fg$ and replace one factor by the corresponding modules of $W_{k}^\rho(g)$ obtained via 
quantum Hamiltonian reduction from Weyl modules. For example, we can define
\[
H_{DS}^{R, \rho} \left(A^n[\fg, \kappa]\right) := \bigoplus_{\lambda \in P^+\cap Q} V_k(\lambda) \otimes H_{DS}^\rho\left(V_{k'}(\lambda)\right).
\]
Here $R$ stands for the reduction on the right factor.
Again, the expectation from gauge theory is that this is a simple vertex algebra for generic $\kappa$. Moreover in some instances this algebra is expected to be isomorphic to a simpler vertex algebra. For example, if we take $n=1$ and $\rho$ the regular embedding of $\mathfrak{sl}_2$ in $\fg$, then main Theorem 3 (2) of \cite{ACL1} implies that $H_{DS}^{R, \rho_{\text{reg}}} \left(A^n[\fg, \kappa]\right) \cong V_{k-1}(\fg) \otimes L_1(\fg)$ as modules for $V_k(\fg) \otimes W_{k'}^{\rho_{\text{reg}}}(\fg)$-modules. Here $ L_1(\fg)$ denotes the simple affine vertex algebra of $\fg$ at level one. 

\item Similarly we can define the quantum Hamiltonian reduction on the left factor $H_{DS}^{L, \rho}$ or on both $H_{DS}^{\rho, \rho'}$, e.g. 
\[
H_{DS}^{\rho, \rho'} \left(A^n[\fg, \kappa]\right) := \bigoplus_{\lambda \in P^+\cap Q} H_{DS}^\rho\left(V_{k}(\lambda)\right) \otimes H_{DS}^{\rho'}\left(V_{k'}(\lambda)\right).
\]

\item The $A^n[\fg, \kappa]$ vertex algebras are expected to have a nice large $\kappa$ limit. The expectation is for example that for $\kappa\rightarrow \infty$ one has as $G\times V_{k'}(\fg)$-modules
\[
\lim_{\kappa\rightarrow \infty} A^n[\fg, \kappa] \cong Z\otimes \bigoplus_{\lambda \in P^+\cap Q} \rho_\lambda \otimes V_{k'}(\lambda),
\]
where $Z$ is a commutative vertex algebra and $\rho_\lambda$ is the irreducible highest-weight representation of $\fg$ (or equivalently the Lie group $G$ of $\fg$) of highest-weight $\lambda$. In other words, one expects that affine vertex algebra action gets traded for an action of $G$ in the large coupling limit. 

These algebras were identified explicitly in \cite{CG} for $\fg=\fsl_2$ and $n=1$. Here we will identify 
them for $\fg=\fsl_2$ and $n=2$. 
\end{itemize}

\subsection{Results}

Among finite dimensional simple Lie superalgebras there is one continuous exceptional family called $d(2, 1;a)$, parameterized by the complex number $a$. The large $N=4$ superconformal algebra is the quantum Hamiltonian reduction of the affine vertex operator superalgebra of $d(2, 1;a)$ at level $k_d$ with respect to a minimal nilpotent element \cite{KW}. There is thus a two parameter family of such $W$-superalgebras whose simple quotient we denote by $\largeVir{k_d}{a}$. Furthermore, the large $N=4$ superconformal algebra has a good $a \to \infty$ limit, which contains a large nontrivial ideal; the quotient by this ideal is the small $N=4$ superconformal algebra $\smallVir{k_d}$.

Recent work \cite{AKFPP, AKFPP2} on conformal embeddings of affine vertex operator algebras in minimal W-superalgebras shows that $\largeVir{1/2}{a}$ is a vertex algebra extension of $V_{-(a+3)/2}(\mathfrak{sl}_2) \otimes V_{-(a^{-1}+3)/2}(\mathfrak{sl}_2)$. This is precisely the choice of levels which appear in 
$A^2[\fsl_2, \kappa= \frac{1-a}{2}]$! This motivates us to study the branching rules of this embedding, i.e. the decomposition of $\largeVir{1/2}{a}$ in terms of $V_{-(a+3)/2}(\mathfrak{sl}_2) \otimes V_{-(a^{-1}+3)/2}(\mathfrak{sl}_2)$-modules. For the same reason, we also study the corresponding decomposition of $\smallVir{1/2}$.

We denote as before by $V_k(\lambda)$ the Weyl-module of $V_{k}(\mathfrak{sl}_2)$ of highest-weight $\lambda$, $\rho_{\lambda}$ the corresponding irreducible highest-weight representation of $SU(2)$  and the fundamental weight of $\mathfrak{sl}_2$ is denoted by $\omega$.  Our first main result is Theorem \ref{thm:smallN4}. It says the following: As a module for $SU(2) \times V_{-3/2}(\mathfrak{sl}_2)$ the simple small $N=4$ superconformal algebra at central charge $-9$ decomposes as 
\[
\smallVir{1/2} \cong \bigoplus_{m=0}^\infty \ \rho_{m\omega} \ \otimes \ V_{-3/2}(m\omega).
\]
 The proof of this decomposition result uses Adamovi\'c's construction of this algebra \cite{Ad1}
as well as conformal embeddings found in \cite{C}.

Analogously, one can now use the theory of deformable families of vertex algebras as introduced in \cite{CL, CL2} to study the decomposition of $\largeVir{1/2}{a}$. For this let $a$ be generic.
Our second main result is then Corollary \ref{cor:larN=4} saying that
as a $V_{-(a+3)/2}(\mathfrak{sl}_2) \otimes V_{-(a^{-1}+3)/2}(\mathfrak{sl}_2)$-module
\[
\largeVir{1/2}{a} \cong \bigoplus_{m=0}^\infty V_{-(a+3)/2}(m\omega) \otimes V_{-(a^{-1}+3)/2}(m\omega).
\]
The quantum Hamiltonian reduction of $V_{k}(\mathfrak{sl}_2)$ is the Virasoro vertex algebra at central charge $13-6(k+2)-6(k+2)^{-1}$. We consider the case $k=-(a^{-1}+3)/2$ and apply the quantum Hamiltonian reduction functor to $\largeVir{1/2}{a}$. Our third result is Theorem \ref{thm:red} stating that this reduction is isomorphic to $V_{\ell}(\mathfrak{osp}(1|2))$ ($\ell=-(a+3)/2$)  and the reduction on both affine vertex subalgebras is isomorphic to the $N=1$ superconformal algebra at central charge $3/2 + 3(a+2+a^{-1})$ times a free fermion \VOSA. 

Let us summarize our results in terms of the language of Section \ref{sec:pred}. Let $\kappa$ be generic.
\begin{enumerate}
\item $$A^2[\mathfrak{sl}_2, \kappa] \cong \left(\largeVir{1/2}{a} \right)_{\text{even}} \qquad  \text{with} \qquad \kappa= (1-a)/2.$$
\item $$H_{DS}^{R, \rho_{\text{reg}}} \left(\largeVir{1/2}{a} \right) \cong V_{\ell}(\mathfrak{osp}(1|2))\qquad  \text{with} \qquad \ell=-(a+3)/2.$$
\item $$H_{DS}^{\rho_{\text{reg}}, \rho_{\text{reg}}} \left(\largeVir{1/2}{a} \right) \cong  \sVir_{N=1,c} \otimes F(1)$$ with $\sVir_{N=1,c}$ the $N=1$ superconformal algebra at central charge $c=3/2 + 3(a+2+a^{-1})$ and $F(1)$ a free fermion \VOSA{} of rank one. 
\item We have an isomorphism of $SU(2) \times V_{-3/2}(\mathfrak{sl}_2)$-modules $$\lim_{a\rightarrow\infty} \largeVir{1/2}{a} \cong Z\otimes \smallVir{1/2} \cong Z\otimes \bigoplus_{m=0}^\infty \ \rho_{m\omega} \ \otimes \ V_{-3/2}(m\omega)$$ with $Z$ a commutative Heisenberg vertex algebra of rank three. 
\end{enumerate}

\subsection{Gauge theory motivation}
We refer to \cite{CG, GR, FG} for a general discussion of the construction of ``corner vertex algebras'', i.e.
vertex algebras supported at 2d junctions between 3d topological boundary conditions in the GL-twisted 4d gauge theory. 

Here we will only need the following pieces of information, which hold as stated for non-rational $\kappa$. 
\begin{enumerate}
\item The 4d gauge theory is labelled by a gauge group $G$ and a topological coupling $\kappa$ (aka $\Psi$ in \cite{KWi, GR, CG}).
The term ``gauge group'' means roughly a choice of global form for the gauge algebra $\fg$. More precisely, it involves some extra data in the 
form of ``discrete theta angles''.
\item The 3d boundary conditions are associated to a spin-ribbon category of ``boundary line defects''. 
\item Consider two boundary conditions associated to categories $\cC_1$ and $\cC_2$. Then a 2d junction between these boundary conditions
is associated to a vertex algebra $A_{12}$ equipped with a functor $F_{12}: \bar \cC_1 \times \cC_2 \to A_{12}-\mathrm{mod}$. 
\item We often denote a junction between boundaries $B_1$ and $B_2$ as an arrow $B_1 \to B_2$. 
If we need to distinguish different junctions, we label the arrow. 
\item Junctions can be composed. The corresponding vertex algebras compose by extension along $\cC_2$:
\begin{equation} A_{13} \equiv (F_{12}\times F_{23})(1 \otimes \mathrm{Diag}_{\cC_2} \otimes 1) \end{equation}
\item There is a ``duality groupoid'' which acts on $\kappa$ by certain fractional linear transformations in $PGL_2(\bZ)$ and 
possibly on $G$. 4d theories related by duality are equivalent, and duality relates boundary conditions and junctions of 
dual pairs of theories. The categories of lines and junction vertex algebras match under duality. 
\item There are two universal families of boundary conditions $N_{p,q}$ and $D_{p,q}$, labelled by coprime integers $p$ and $q$ defined up to overall rescaling. 
The duality groupoid acts on the $(q,p)$ label as a column vector for $PGL_2(\bZ)$.
\item The category of lines at $N_{1,0}$ is essentially the Kazhdan-Lusztig category $KL_\kappa(G)$, 
up to subtleties associated to discrete theta angles. 
\item The category of lines at $D_{0,1}$ is essentially $D_\kappa-\mathrm{mod}(Gr_G)$, 
up to subtleties associated to discrete theta angles. We will not use this fact directly, but we will use the related observation that 
nice junctions of the form $D_{0,1} \to B$ support vertex algebras with a $V_{\kappa + n}(\fg)$ subalgebra with integer $n$ 
and admit a corresponding $G[[z]]$ action. 
\item There is an $N_{0,1} \to N_{1,0}$ junction supporting the ${\mathcal W}_\kappa(\fg)$ W-algebra, up to subtleties associated to discrete theta angles.
\item There is an $D_{0,1} \to N_{1,0}$ junction supporting the $\fg_\kappa:= V_{\kappa-h^\vee}(\fg)$  Kac-Moody algebra, up to subtleties associated to discrete theta angles.
\end{enumerate}
We denote with the greek symbol $\kappa$ the critically shifted level of algebras and with latin $k$ the non-critically shifted one.  

If we specialize to $\fg = \fsl_2$ and ignore subtleties concerning the global form of the group, the following junction vertex algebras have been identified \cite{CG, FG}: 
\begin{itemize}
\item $N_{0,1} \to N_{1,0}$: $\Vir_\kappa$. This is the Virasoro algebra of central charge $c_{\Vir_\kappa} = 13- 6 \kappa - 6 \kappa^{-1}$. 
\item $N_{0,1} \to N_{2,1}$: $\sVir_{2 \kappa-1}$. The $N=1$ super-Virasoro algebra $\sVir_\kappa$ has central charge $c_{\sVir_\kappa} = \frac{15}{2}- 3 \kappa - 3 \kappa^{-1}$. 
\item $D_{0,1} \to N_{1,0}$: $\fsl(2)_\kappa= V_{\kappa-2}(\fsl_2)$. This is the Kac-Moody algebra at critically shifted level $\kappa$. It has central charge $c_{\fsu(2)_\kappa} = 3-6 \kappa^{-1}$.
\item $D_{0,1} \to N_{2,1}$: $\fosp(1|2)_{2\kappa-1}$. This super-Kac-Moody algebra has an $\fsl(2)_{\frac{1+\kappa}{2}}$ subalgebra. 
It has central charge $c_{\fosp(1|2)_{2\kappa-1}} =1-3 \kappa^{-1}$.
\item $D_{0,1} \to D_{1,0}$: $L_1(d(2,1;-\kappa))$. This is a quotient of the super-Kac-Moody
  algebra based on the $d(2,1;-\kappa)$ exceptional
  superalgebra which is an extension of $\fsl(2)_{\kappa+1}
  \otimes \fsl(2)_{\kappa^{-1}+1} \otimes L_1(\fsu(2))$. The central charge of this vertex algebra is
  $c_{d(2,1;-\kappa)_1} = 1$. 
\end{itemize}

These algebras enjoy many coset relations. They can all be understood in terms of 
composition of junctions. These algebras also enjoy several relations based on quantum Hamiltonian reductions.
These relations map $D_{0,1} \to B$ junction vertex algebras to $N_{1,n} \to B$ junction vertex algebras,
for appropriate $n$. 

The main result of this paper is the identification of one more junction vertex algebra: 
\begin{itemize}
\item $D_{0,1} \to D_{2,-1}$: $\largeVir{1/2}{-1-2 \kappa}$. This is the large $N=4$ superconformal algebra at central charge $-6$, 
which is an extension of $\fsl(2)_{\kappa+1} \otimes \fsl(2)_{\frac{\kappa+1}{2 \kappa+1}}$
\end{itemize}
The extension is compatible with the composition of junctions $D_{0,1} \to N_{1,-1} \to D_{2,-1}$. Indeed, 
$D_{0,1} \to N_{1,-1}$ supports $\fsl(2)_{\kappa+1}$ and 
the $N_{1,-1} \to D_{2,-1}$ junction at coupling $\kappa$ is equivalent to the $D_{1,-2} \to N_{1,-1}$ junction 
at coupling $\kappa^{-1}$, which is equivalent to the $D_{1,0} \to N_{1,1}$ junction 
at coupling $2+\kappa^{-1}$ and to the $D_{0,1} \to N_{1,-1}$ junction 
at coupling $-\frac{\kappa}{2\kappa+1}$.

We also verify that this statement is compatible with the reduction of $D_{0,1} \to D_{2,1}$ to $D_{0,1} \to N_{2,1}$ and $N_{0,1} \to N_{2,1}$ junctions 
under quantum Hamiltonian reductions.

\subsection{The $\kappa \to \infty$ limit}
The gauge theory structures at rational $\kappa$ and $\kappa = \infty$ are somewhat more subtle than at generic $\kappa$. 
Still, one expects the junction vertex algebras to have reasonably good limits as $\kappa$ is brought to rational points,
though some fields such as Kac-Moody currents of infinite level or Virasoro modes of infinite central charge
may have to be rescaled and replaced by classical fields in the limit. 

The $\kappa \to \infty$ limit of the above table gives 
\begin{itemize}
\item $N_{0,1} \to N_{1,0}$: $\Vir_\infty$, classical Virasoro stress tensor. 
\item $N_{0,1} \to N_{2,1}$: Semi-classical super-Virasoro $\sVir_\infty$. The rescaled super-conformal generator 
survives the limit, but the OPE contains a classical stress tensor. 
\item $D_{0,1} \to N_{1,0}$: $\fsl(2)_\infty$, a classical $\mathfrak{sl}_2$ connection.
\item $D_{0,1} \to N_{2,1}$: $\fosp(1|2)_\infty$. The two fermionic generators survive as $\mathfrak{psu}(1|1)$
currents, with OPE deformed by a classical $\mathfrak{sl}(2)$ connection.
\item $D_{0,1} \to D_{1,0}$: $L_1(\mathfrak{psu}(2|2))$, with OPE deformed by a classical $\mathfrak{sl}(2)$ connection.
\end{itemize}
and our new entry becomes 
\begin{itemize}
\item $D_{0,1} \to D_{2,-1}$: $\smallVir{1/2}$, the small $N=4$ super-Virasoro of central charge $-9$, with OPE deformed by
a classical $\mathfrak{sl}(2)$ connection. 
\end{itemize}

\subsection{Gauge theory predictions from VOA}\label{sec:predtwo}
Somewhat surprisingly, this is not the first gauge theory appearance of the small $N=4$ super-Virasoro of central charge $-9$.
This vertex algebra was found, as well, as the chiral algebra associated to four-dimensional ${\cal N}=4$ gauge theory, in the sense of 
\cite{BLLPRR}. Concretely, it emerges as the result of a BRST reduction of symplectic bosons valued in the adjoint of the gauge group. 
The $\mathrm{sl}(2)$ global symmetry emerges in a surprising way and is not manifest in the BRST construction. 

The chiral algebras which appear in \cite{BLLPRR} can always be given an alternative physical construction, as boundary 
vertex algebras for certain three-dimensional gauge theories \cite{CoG}. In the case at hand, that would be the 
maximally supersymmetric (i.e. ${\cal N}=8$) three-dimensional $SU(2)$ gauge theory. The theory has an 
$\mathrm{so}(7)$ R-symmetry, which is enhanced to $\mathrm{so}(8)$ in the IR. 

The boundary supporting the chiral algebra preserves an 
$\mathrm{so}(3)_C \times \mathrm{so}(4)$ subgroup in $\mathrm{so}(7)$, which is enhanced to $\mathrm{so}(4) \times \mathrm{so}(4)$ in the IR. 
More precisely, the second $\mathrm{so}(4)$ factor includes the $\mathrm{su}(2)_H$ R-symmetry and 
an $\mathrm{su}(2)_F$ flavor symmetry. The first $\mathrm{so}(4)$ factor includes the $\mathrm{su}(2)_{C'}$ R-symmetry and 
the $\mathrm{su}(2)_{F'}$ flavor symmetry which emerge in the IR. 

Correspondingly, boundary chiral algebra itself is built though the BRST reduction of symplectic bosons,
which only makes manifest the $\mathfrak{sl}(2)$ current algebra associated to the $\mathrm{su}(2)_F$
flavor symmetry. The emergent $\mathrm{sl}(2)$ global symmetry is associated to the emergent 
$\mathrm{su}(2)_{F'}$ flavor symmetry.

Crucially, the kernel vertex algebras at $\kappa = \infty$ are expected to coincide
with the boundary vertex algebras of three-dimensional (i.e. ${\cal N}=4$) gauge theories 
which appear at ``duality walls''. For example, the $\kappa = \infty$ limit of $L_1(d(2,1;-\kappa))$
coincides with the boundary vertex algebra of the ``$T[SU(2)]$'' gauge theory. 

Hence we arrive at the following conjectures: 
\begin{itemize}
\item The $ST^2S$ duality wall in four-dimensional $SU(2)$ SYM supports the three-dimensional 
${\cal N}=8$ SCFT defined by three-dimensional $SU(2)$ gauge theory. 
\item The 3d SCFT is treated as 
an ${\cal N}=4$ theory, with $\mathrm{su}(2)_F \times \mathrm{su}(2)_{F'}\subset \mathrm{so}(8)_R$ flavour symmetry.
\item The 3d SCFT is coupled to the four-dimensional gauge theories on the two sides of the wall 
by gauging $\mathrm{su}(2)_F \times \mathrm{su}(2)_{F'}$.
\end{itemize}

It would be nice to test these conjectures further. This could be done, for example, by 
computing the action of 't Hooft-Wilson lines on the duality wall as in \cite{DGV}.
This should parallel a similar check about the action of Hecke modifications on 
conformal blocks of kernel vertex algebras, which is also an open problem even for $n=1$. 
\\

\noindent{\bf Acknowledgements}
TC thanks Tomoyuki Arakawa for discussions on related issues. He is supported by NSERC discovery grant $\#$RES0020460. 
DG is supported by the NSERC Discovery Grant program and by the Perimeter Institute for Theoretical Physics. Research at Perimeter Institute is supported by the Government of Canada through Industry Canada and by the Province of Ontario through the Ministry of Research and Innovation.
AL is supported by Simons Foundation Grant \#318755.

\section{The large $N=4$ superconformal algebra}
The large $N=4$ superconformal algebra $V(k,a)$ arises as a quantum Hamiltonian reduction of the affine vertex operator superalgebra of the exceptional Lie superalgebra $d(2,1; a)$ at level $k$ \footnote{The reader should be warned that $k$ here is $k_d$ in the introduction, and has nothing to do with the $\kappa$ of the gauge theory,
which is instead related to $a$.}. It depends on two complex parameters $k$ and $a$, and has strong generators $e, f, h, e', f', h', G^{\pm\pm}, L$. Here $L$ is a Virasoro field of central charge $$c=-6k-3,$$ $e, f, h, e', f', h'$ are even primary fields of weight $1$, and $G^{\pm\pm}$ are odd primary of weight $3/2$. The operator product algebra appears explicitly in \cite{KW}. First, $\{e,f,h\}$ and $\{e',f',h'\}$ generate two commuting affine $\mathfrak{sl}(2)$ algebras at levels $- \frac{a+1}{a}k-1$ and $-(a+1)k-1$, respectively.
\begin{equation}\nonumber
\begin{split}
h'(z) h'(w) & \sim  -2((a+1)k+1)(z-w)^{-2}, \qquad
h(z) h(w) \sim -2 \left(\frac{a+1}{a}k+1\right)(z-w)^{-2},\\
e' (z) f'(w) & \sim - ((a+1)k+1)(z-w)^{-2} + h'(w)(z-w)^{-1}, \\
e(z) f(w) & \sim -  \left(\frac{a+1}{a}k+1\right)(z-w)^{-2} + h(w)(z-w)^{-1},\\
h' (z) e'(w) &\sim 2e'(w)(z-w)^{-1} \qquad
h(z) e(w) \sim 2e(w)(z-w)^{-1}, \\
h' (z) f'(w) & \sim -2f'(w)(z-w)^{-1}, \qquad
h(z) f(w) \sim  -2f(w)(z-w)^{-1}.
\end{split}
\end{equation}
Next, these act on the odd fields $G^{\pm \pm}$ by
\begin{equation}\nonumber
\begin{split}
h'(z) G^{\pm\pm}(w) & \sim \pm  G^{\pm\pm}(w)(z-w)^{-1}, \qquad
h' (z) G^{\pm\mp}(w) \sim \pm  G^{\pm\mp}(w)(z-w)^{-1}, \\
h (z)  G^{\pm\pm}(w) &\sim \pm  G^{\pm\pm}(w)(z-w)^{-1}, \qquad
h (z)  G^{\pm\mp}(w) \sim \mp  G^{\pm\mp}(w)(z-w)^{-1}, \\
e' (z)  G^{- -}(w) & \sim -  G^{+ -}(w)(z-w)^{-1}, \qquad
e' (z)  G^{- +}(w) \sim -  G^{+ +}(w)(z-w)^{-1}, \\
e (z)  G^{- -}(w) & \sim   G^{- +}(w)(z-w)^{-1}, \qquad
e (z)  G^{+ -}(w) \sim   G^{+ +}(w)(z-w)^{-1}, \\
f' (z)  G^{+ +}(w) & \sim - G^{- +}(w)(z-w)^{-1}, \qquad
f' (z) G^{+ -}(w) \sim  -  G^{- -}(w)(z-w)^{-1}, \\
f (z) G^{+ +}(w) & \sim   G^{+ -}(w)(z-w)^{-1}, \qquad
f (z) G^{- +}(w)  \sim   G^{- -}(w)(z-w)^{-1}, \\
\end{split}
\end{equation}
Finally, the OPEs of the odd fields are 
\begin{equation}\nonumber
\begin{split}
G^{+ +}(z)   G^{+ +}(w) &  \sim   \frac{2a}{(a+1)^2} (:e'e:) (w)(z-w)^{-1}, \qquad
G^{- -} (z) G^{- -}(w) \sim  \frac{2a}{(a+1)^2} (:f' f:)(w)(z-w)^{-1}, \\ 
G^{- +}(z)  G^{- +}(w) & \sim  -\frac{2a}{(a+1)^2} (:f'e:)(w)(z-w)^{-1}, \qquad
G^{+ -}(z) G^{+ -}(w) \sim   -\frac{2a}{(a+1)^2} (:e'f:)(w)(z-w)^{-1}, \\ 
G^{+ +}(z)  G^{- +}(w) & \sim -  \frac{2a}{a+1}\bigg(\frac{1}{a+1} + k \bigg) e(w)(z-w)^{-2} \\ & + \bigg(\frac{a}{(a+1)^2} :h'e: -  \frac{a}{a+1}\bigg(\frac{1}{a+1} + k \bigg) \partial e\bigg)(w)(z-w)^{-1},
\end{split}
\end{equation}
\begin{equation}\nonumber
\begin{split}
G^{+ +}(z) G^{+ -}(w) & \sim   \frac{2}{a+1}\bigg(\frac{a}{a+1} + k \bigg) e'(w)(z-w)^{-2} \\ & + \bigg(-\frac{a}{(a+1)^2} :he': +  \frac{1}{a+1}\bigg(\frac{a}{a+1} + k \bigg) \partial e'\bigg)(w)(z-w)^{-1},
\end{split}
\end{equation}
\begin{equation}\nonumber
\begin{split}
G^{- -}(z) G^{- +}(w) & \sim   \frac{2}{a+1}\bigg(\frac{a}{a+1} + k \bigg) f'(w)(z-w)^{-2} \\ & + \bigg(\frac{a}{(a+1)^2} :hf': +  \frac{1}{a+1}\bigg(\frac{a}{a+1} + k \bigg) \partial f'\bigg)(w)(z-w)^{-1},\\
G^{- -}(z)  G^{+ -}(w) & \sim -  \frac{2a}{a+1}\bigg(\frac{1}{a+1} + k \bigg) f(w)(z-w)^{-2} \\ & + \bigg(-\frac{a}{(a+1)^2} :h'f: -  \frac{a}{a+1}\bigg(\frac{1}{a+1} + k \bigg) \partial f\bigg)(w)(z-w)^{-1},
\end{split}
\end{equation}
\begin{equation}\nonumber
\begin{split}
G^{+ +}(z)  G^{- -}(w)  & \sim   - 2 \left(k(k+1) +\frac{a}{(a+1)^2}\right)(z-w)^{-3} + \bigg(\frac{a + k + a k}{(1 + a)^2} h' +  \frac{a (1 + k + a k)}{(1 + a)^2} h\bigg)(w)(z-w)^{-2} \\  
& + \bigg(k L +\frac{a}{4 (1 + a)^2} :h' h':  + \frac{a}{4 (1 + a)^2} :hh:  - \frac{a}{2 (1 + a)^2} :h h':   + \frac{a}{(1 + a)^2} :e' f':   \\ 
& + \frac{a}{(1 + a)^2} :ef:  + \frac{a k}{2 (1 + a)} \partial h   + \frac{k}{2 (1 + a)} \partial h' \bigg)(w)(z-w)^{-1}\\
G^{- +}(z)  G^{+ -}(w) &  \sim   2 \left(k(k+1) +\frac{a}{(a+1)^2}\right)(z-w)^{-3} + \bigg(\frac{a + k + a k}{(1 + a)^2} h' -  \frac{a (1 + k + a k)}{(1 + a)^2} h\bigg)(w)(z-w)^{-2} \\ 
& + \bigg(-k L -\frac{a}{4 (1 + a)^2} :h' h':  - \frac{a}{4 (1 + a)^2} :hh:  - \frac{a}{2 (1 + a)^2} :h h':   - \frac{a}{(1 + a)^2} :e' f':   \\ 
& - \frac{a}{(1 + a)^2} :ef:  - \frac{a k}{2 (1 + a)} \partial h   + \frac{2 a + k + a k}{2 (1 + a)^2} \partial h' \bigg)(w)(z-w)^{-1}\\
\end{split}
\end{equation}
We shall denote by $\largeVir{k}{a}$ the simple quotient of $V(k,a)$ by its maximal proper ideal $\cI_{k,a}$ graded by conformal weight. For generic values of $k$ and $a$, $V(k,a)$ is simple so that $V(k,a) = \largeVir{k}{a}$. 


It will be convenient to introduce a change of variables and replace the Virasoro field $L$ with the field 
$$L^C = L - L^{\mathfrak{sl}'_2},\qquad L^{\mathfrak{sl}'_2} = -\frac{1}{4(-1+k+a k)} \bigg(:h' h': + 2 :x' y':  + 2 :y' x': \bigg).$$
Note that $L^C$ is just the Virasoro field for the coset $$\text{Com}(V^{-(a+1)k-1} (\mathfrak{sl}_2), V(k,a)),$$ where $V^{-(a+1)k-1}(\mathfrak{sl}_2) \subseteq V(k,a)$ is the subVOA generated by $\{x', y', h'\}$. The central charge of $L^C$ is  $-\frac{6 k (a + k + a k)}{-1 + k + a k}$. With this change of variables, the fields $G^{\pm \pm}$ are no longer primary, but now satisfy
\begin{equation}
\begin{split} L^C(z) G^{++}(w) & \sim \frac{3  (-1 + 2 k + 2 a k)}{4 (-1 + k + a k)} G^{++}(w)(z-w)^{-2} 
\\ & + \bigg(\frac{1}{2 (-1 + k + a k)} :h' G^{++}: + \frac{1}{1 - k - a k} :x' G^{-+}: + \partial G^{++}\bigg)(w)(z-w)^{-1},\end{split} \end{equation}
 
\begin{equation}
\begin{split} L^C(z) G^{+-}(w) & \sim \frac{3 (-1 + 2 k + 2 a k)}{4(-1 + k + a k)} G^{+-}(w)(z-w)^{-2} 
\\ & + \bigg(\frac{1}{2 (-1 + k + a k)} :h' G^{+-}: + \frac{1}{1 - k - a k} :x' G^{--}: + \partial G^{+-}\bigg)(w)(z-w)^{-1},\end{split} \end{equation}

\begin{equation}
\begin{split} L^C(z) G^{--}(w) & \sim \frac{3  (-1 + 2 k + 2 a k)}{4 (-1 + k + a k)} G^{--}(w)(z-w)^{-2} 
\\ & + \bigg(-\frac{1}{2 (-1 + k + a k)} :h' G^{--}: + \frac{1}{1 - k - a k} :y' G^{+-}: + \partial G^{--}\bigg)(w)(z-w)^{-1},\end{split} \end{equation}

\begin{equation}
\begin{split} L^C(z) G^{-+}(w) & \sim \frac{3  (-1 + 2 k + 2 a k)}{4 (-1 + k + a k)} G^{-+}(w)(z-w)^{-2} 
\\ & + \bigg(-\frac{1}{2 (-1 + k + a k)} :h' G^{-+}: + \frac{1}{1 - k - a k} :y' G^{++}: + \partial G^{-+}\bigg)(w)(z-w)^{-1},\end{split} \end{equation}

Additionally, the OPEs $G^{++}(z) G^{--}(w)$ and $G^{-+}(z) G^{+ -}(w)$ are replaced with

\begin{equation}
\begin{split}
G^{+ +}(z)  G^{- -}(w)  & \sim   - 2 \left(k(k+1) +\frac{a}{(a+1)^2}\right)(z-w)^{-3} 
\\ & + \bigg(\frac{a + k + a k}{(1 + a)^2} h' +  \frac{a (1 + k + a k)}{(1 + a)^2} h\bigg)(w)(z-w)^{-2} 
\\ & + \bigg(k L^C - \frac{a + k + a k}{4 (1 + a)^2 (-1 + k + a k)} :h' h':  
+ \frac{a}{4 (1 + a)^2} :hh:  - \frac{a}{2 (1 + a)^2} :h h': 
\\ &  - \frac{a + k + a k}{(1 + a)^2 (-1 + k + a k)} :e' f':   + \frac{a}{(1 + a)^2} :ef:  + \frac{a k}{2 (1 + a)} \partial h   
\\ & + \frac{k (a + k + a k)}{2 (1 + a) (-1 + k + a k)} \partial h' \bigg)(w)(z-w)^{-1},\end{split}
\end{equation}

\begin{equation}
\begin{split} G^{- +}(z)  G^{+ -}(w) &  \sim   2 \left(k(k+1) +\frac{a}{(a+1)^2}\right)(z-w)^{-3} 
\\ & + \bigg(\frac{a + k + a k}{(1 + a)^2} h' -  \frac{a (1 + k + a k)}{(1 + a)^2} h\bigg)(w)(z-w)^{-2} \\ 
& + \bigg(-k L^C + \frac{a + k + a k}{4 (1 + a)^2 (-1 + k + a k)} :h' h':  - \frac{a}{4 (1 + a)^2} :hh:  - \frac{a}{2 (1 + a)^2} :h h': 
\\ & + \frac{a + k + a k}{(1 + a)^2 (-1 + k + a k)} :e' f':  - \frac{a}{(1 + a)^2} :ef:  - \frac{a k}{2 (1 + a)} \partial h   \\ & + \frac{(-2 + k + a k) (a + k + a k)}{2 (1 + a)^2 (-1 + k + a k)} \partial h' \bigg)(w)(z-w)^{-1}\\
\end{split}
\end{equation}

From now on, we shall replace $L$ with $L^C$, and use these OPE relations instead of the original ones. We may regard $k$ and $a$ as formal variables, and we regard $V(k,a)$ as a vertex algebra over the ring $R$ consisting of rational functions in $k$ and $a$ with possible poles at $a = 0$, $a+1=0$, and $1 - k - a k = 0$. Since $V(k,a)$ is generically simple, $V(k,a)$ is simple as a vertex algebra over $R$. Also, it is freely generated as a vertex superalgebra, so each weight graded space $V(k,a)[n]$ for $n \in \frac{1}{2} \mathbb{Z}_{\geq 0}$, is a free $R$-module.

\subsection{$\largeVir{k}{a}$ as deformable families}
For all $n\in \frac{1}{2} \mathbb{Z}_{\geq 0}$, $V(k,a)[n]$ has the Shapovalov pairing $\langle, \rangle_n$ given by $$\langle \omega, \nu \rangle = \omega_{(2n-1)} \nu,$$ which takes values in $R$. If $p=p(k,a)$ is an irreducible factor of the determinant of $\langle, \rangle_n$ for some $n$, there will be a nontrivial null vector in $V(k,a)[n]$. Let $I_p \subset R$ be the ideal generated by $p$. Since $R \cong V(k,a)[0]$, we can regard $I_p$ as a subset of $V(k,a)[0]$. Let $$I_p \cdot V(k,a) \subset V(k,a)$$ be the set of $I_p$-linear combinations of elements of $V(k,a)$, which is just the vertex algebra ideal generated by $I_p$. Then the quotient
$$V^{I_p}(k,a) = V(k,a) / (I_p\cdot V(k,a))$$ is a vertex algebra defined over the ring $R_p = R/I_p$. It is a free $R_p$-module, but it is is not simple as a vertex algebra over $R_p$ since its maximal proper graded ideal $\cI_p$ contain all null vectors of the Shapovalov pairing. For each $n$, the weight $n$ component $\cI_p[n]$ is $R_p$-submodule of $V^{I_p} (k,a)[n]$, and the simple quotient
$$L^{I_p}(k,a) = V^{I_p}(k,a) / \cI_p$$ is again a vertex algebra over $R_p$. However, $\cI_p[n]$ need not be a direct summand of $V^{I_p} (k,a)[n]$, and the quotient need not be free.

Given a multiplicatively closed subset $D \subset R_p$, let $D^{-1}R_p$ denote the localization of $R_p$ along $D$, and consider the following localizations of $R_p$-modules 
\begin{equation}
\begin{split}
D^{-1} V^{I_p}(k,a) &= (D^{-1}R_p) \otimes_{R_p} V^{I_p}(k,a),\\
 D^{-1} \cI_p &= (D^{-1}R_p) \otimes_{R_p} \cI_p,\\
 D^{-1} L^{I_p}(k,a) &= (D^{-1}R_p) \otimes_{R_p} L^{I_p}(k,a).
\end{split}
\end{equation} 

\begin{lemma} There exists a multiplicatively closed subset $D \subset R_p$ which is at most countably generated, such that $D^{-1} L^{I_p}(k,a)$ is a free $D^{-1}R_p$-module.
\end{lemma}

\begin{proof} For each $n$, since $V^{I_p} (k,a)[n]$ is a free $R_p$-module of finite rank, there exists a finitely generated multiplicative set $D_n \subset R_p$ with the following property: $D_n^{-1} V^{I_p} (k,a)[n]$ has a splitting $$D_n^{-1} V^{I_p}(k,a)[n] = D^{-1} \cI_p[n] \oplus C_n,$$ where $C_n$ is a complementary $D^{-1} R_p$-submodule, and both summands are free $D^{-1} R_p$-modules.

Taking $D = \bigcup D_n$, we see that $D$ is at most countably generated, and as a $D^{-1}R_p$-module,
$$D^{-1} V^{I_p} (k,a)  = \bigg(\bigoplus_{n\geq 0} C_n\bigg) \bigoplus D^{-1} \cI_p,$$ so that $$D^{-1} L^{I_p} (k,a)  = \bigoplus_{n\geq 0} C_n,$$ and in particular is free.
\end{proof}

In this paper, we only need the case where $p(a,k) = k - k_0$ for some fixed level $k_0$. In this case, $R_p$ is isomorphic to the ring of rational functions in $a$ with possible poles at $a = 0$, $a=-1$, and $ a = \frac{1-k_0}{k_0}$, which we denote by $\mathbb{C}[a]_{\{0,-1, (1-k_0)/k_0\}}$. In this case, we denote the simple quotient $L^{I_p}(k,a)$ by $\largeVir{k_0}{a}$. 

More generally, for a subset $A\subset \mathbb{C}$, let $\mathbb{C}[a]_A$ denote the ring of rational functions in $a$ with possible poles in $A$. As special case of the above result, we have 

\begin{corol} \label{cor:deformfam}
For any $k_0\in \mathbb{C}$, there exists a subset $A\subset \mathbb{C}$ which is at most countable, such that $$\mathbb{C}[a]_A \otimes_{\mathbb{C}[a]_{\{0,-1, (1-k_0)/k_0\}}} \largeVir{k_0}{a}$$ is a free $\mathbb{C}[a]_A$-module. 
\end{corol}

Note that if we rescale $e', f', h'$ by a factor of $1/a$, $V(k,a)$ is defined over the ring of rational functions in $a$ and $k$ of degree at most zero in $a$, with possible poles at $a=0$, $a=-1$, and $a = \frac{1 - k}{k}$. Likewise, with this rescaling, $\largeVir{k_0}{a}$ is defined over the ring $F$ of rational functions in $a$ of degree at most zero, with possible poles at $a = 0$, $a=-1$, or $a = \frac{1 - k_0}{k_0}$ Similarly, given a subset $A\subset \mathbb{C}$, we denote by $F_A$ the ring of rational functions in $a$ of degree at most zero, with possible poles in $A$.  In the language of \cite{CL}, Corollary \ref{cor:deformfam} implies that $F_A \otimes_F \largeVir{k_0}{a}$ is a {\it deformable family} of vertex superalgebras. In particular, its graded character is independent of $a$, and coincides with the graded character of the limit $$\largeVir{k_0}{\infty} = \lim_{a \rightarrow \infty} \largeVir{k_0}{a},$$ which is a well-defined vertex algebra over $\mathbb{C}$.

\subsection{The small $N=4$ superconformal algebra as a limit of $\largeVir{k}{a}$}
Regarding $k$ as fixed, and rescaling $e', f', h'$ by $1/ a$ as above, observe that in the limit $\largeVir{k}{\infty}$, the fields $e', f', h'$ become central, and $L^C, e,f,h,G^{\pm \pm}$ satisfy the following OPEs.

\begin{equation}\label{eq:limitsmall 1}
\begin{split}
h(z) h(w) & \sim -2(k+1)(z-w)^{-2},\qquad e(z) f(w)  \sim - (k+1)(z-w)^{-2} + h(w)(z-w)^{-1},\\
h(z) e(w) & \sim 2e(w)(z-w)^{-1}, \qquad h(z) f(w) \sim  -2f(w)(z-w)^{-1}.
\end{split}
\end{equation}

\begin{equation}\label{eq:limitsmall 2}
\begin{split}
h (z)  G^{\pm\pm}(w) &\sim \pm  G^{\pm\pm}(w)(z-w)^{-1}, \qquad
h (z)  G^{\pm\mp}(w) \sim \mp  G^{\pm\mp}(w)(z-w)^{-1}, \\
e (z)  G^{- -}(w) & \sim   G^{- +}(w)(z-w)^{-1}, \qquad
e (z)  G^{+ -}(w) \sim   G^{+ +}(w)(z-w)^{-1}, \\
f (z) G^{+ +}(w) & \sim   G^{+ -}(w)(z-w)^{-1}, \qquad
f (z) G^{- +}(w)  \sim   G^{- -}(w)(z-w)^{-1}, \\
\end{split}
\end{equation}

\begin{equation} \label{eq:limitsmall 3}
\begin{split} 
L^C(z) G^{++}(w) & \sim \frac{3}{2} G^{++}(w)(z-w)^{-2} + \bigg(\frac{1}{2k} :h' G^{++}: -\frac{1}{k} :x' G^{-+}: + \partial G^{++}\bigg)(w)(z-w)^{-1},\\
L^C(z) G^{+-}(w) & \sim \frac{3}{2} G^{+-}(w)(z-w)^{-2} + \bigg(\frac{1}{2 k} :h' G^{+-}: - \frac{1}{k} :x' G^{--}: + \partial G^{+m}\bigg)(w)(z-w)^{-1},\\
L^C(z) G^{--}(w) & \sim \frac{3}{2} G^{--}(w)(z-w)^{-2} + \bigg(-\frac{1}{2 k} :h' G^{--}: - \frac{1}{k} :y' G^{+-}: + \partial G^{--}\bigg)(w)(z-w)^{-1},\\
L^C(z) G^{-+}(w) & \sim \frac{3}{2} G^{-+}(w)(z-w)^{-2} + \bigg(-\frac{1}{2 k} :h' G^{-+}: - \frac{1}{k} :y' G^{++}: + \partial G^{-+}\bigg)(w)(z-w)^{-1},
 \end{split} \end{equation}

\begin{equation} \label{eq:limitsmall 4}
\begin{split}
G^{+ +}(z)   G^{+ +}(w) &  \sim  2 (:e'e:) (w)(z-w)^{-1}, \qquad G^{- -} (z) G^{- -}(w) \sim 2 (:f' f:)(w)(z-w)^{-1}, \\ 
G^{- +}(z)  G^{- +}(w) & \sim  -2 (:f'e:)(w)(z-w)^{-1}, \qquad G^{+ -}(z) G^{+ -}(w) \sim  -2 (:e'f:)(w)(z-w)^{-1}, \\ 
G^{+ +}(z)  G^{- +}(w) & \sim -  2 k e(w)(z-w)^{-2} + \big( :h'e: -  k \partial e\big)(w)(z-w)^{-1},\\
G^{+ +}(z) G^{+ -}(w) & \sim  2(1+k) e'(w)(z-w)^{-2} + \big(- :he': +  (1+k) \partial e'\big)(w)(z-w)^{-1},\\
G^{- -}(z) G^{- +}(w) & \sim  2(1+k) f'(w)(z-w)^{-2}  + \big( :hf': +  (1+k) \partial f'\big)(w)(z-w)^{-1},\\
G^{- -}(z)  G^{+ -}(w) & \sim - 2k f(w)(z-w)^{-2} + \big(- :h'f: - k \partial f\big)(w)(z-w)^{-1},\\
G^{+ +}(z)  G^{- -}(w)  & \sim   - 2 k(k+1)(z-w)^{-3}  + \big((k+1)h'+ k\ h  \big)(w)(z-w)^{-2}
\\ & + \bigg(k L^C -  \frac{k+1}{4k} :h' h': - \frac{k+1}{k} :e' f':  - \frac{1}{2} :h h': + \frac{k}{2} \partial h + \frac{k+1}{2} \partial h'  \bigg)(w)(z-w)^{-1},
\\ G^{- +}(z)  G^{+ -}(w) &  \sim   2 k(k+1) (z-w)^{-3}  + \big((k+1)h' -  k\ h\big)(w)(z-w)^{-2}  
\\ & + \bigg(-k L^C  +  \frac{k+1}{4k} :h' h': + \frac{k+1}{k} :e' f': - \frac{1}{2} :h h': - \frac{k}{2} \partial h + \frac{k+1}{2} \partial h'   \bigg)(w)(z-w)^{-1}.
\end{split}
\end{equation}
Additionally, $L^C$ has central charge $c = -6(k+1)$. For fixed $k$, let $Z\subset \largeVir{k}{\infty}$ denote the vertex algebra generated by  $e', f', h'$, which is a commutative Heisenberg algebra of rank $3$. Moreover, since $Z$ is central in  $\largeVir{k}{\infty}$, it generates a vertex algebra ideal $\cI_{Z}$, and the quotient $$\mathcal{V}^k = \largeVir{k}{\infty} / \cI_{Z},$$ is a vertex algebra with strong generators $L^C,e,f,h,G^{\pm\pm}$.
From the above calculations, we obtain

\begin{lemma} \label{lemma:ainftylimit} We have an exact sequence of vertex algebras 
$$ 0 \rightarrow \cI_{Z}  \rightarrow \largeVir{k}{\infty} \rightarrow \mathcal{V}^k \rightarrow 0.$$ Moreover, the OPE relations \eqref{eq:limitsmall 1}-\eqref{eq:limitsmall 2} hold in $\mathcal{V}^k$, but \eqref{eq:limitsmall 3} and \eqref{eq:limitsmall 4} are replaced with

\begin{equation} \label{eq:limitsmall 5}
\begin{split}
L^C(z) G^{\pm \pm }(w) & \sim \frac{3}{2} G^{\pm \pm}(w)(z-w)^{-2} +  \partial G^{\pm \pm} (w)(z-w)^{-1},\\
G^{+ +}(z)   G^{+ +}(w) &  \sim  0, \qquad G^{- -} (z) G^{- -}(w) \sim 0, \\ 
G^{- +}(z) G^{- +}(w) & \sim  0, \qquad G^{+ -}(z) G^{+ -}(w) \sim  0, \\ 
G^{+ +}(z) G^{+ -}(w) & \sim  0,\qquad G^{- -}(z) G^{- +}(w)  \sim  0,\\
G^{+ +}(z)  G^{- +}(w) & \sim -  2k \ e(w)(z-w)^{-2} - k \ \partial e(w)(z-w)^{-1},\\ 
G^{- -}(z)  G^{+ -}(w) & \sim -  2k\ f(w)(z-w)^{-2} - k \ \partial f(w)(z-w)^{-1},\\
G^{+ +}(z)  G^{- -}(w)  & \sim   - 2 k(k+1)(z-w)^{-3}  + k\ h  (w)(z-w)^{-2} + \big(k L^C  + \frac{k}{2} \partial h \big)(w)(z-w)^{-1},
\\ G^{- +}(z)  G^{+ -}(w) &  \sim   2 k(k+1) (z-w)^{-3}   -  k\ h(w)(z-w)^{-2}  + \big(-k L^C  : - \frac{k}{2} \partial h  \big)(w)(z-w)^{-1}.
\end{split}
\end{equation}
Therefore the generators $L^C,e,f,h,G^{\pm\pm}$ of $\mathcal{V}^k$ satisfy the OPE relations of the small $N=4$ superconformal algebra at level $k$. \end{lemma}


It is not apparent whether or not $\mathcal{V}^k$ is simple, i.e., $\mathcal{V}^k \cong \smallVir{k}$. The main example we need is the case $k = 1/2$, which was studied by Adamovi\'c \cite{Ad1} using a free field realization. Later, we will see that $\mathcal{V}^{1/2}$ is indeed simple.

\begin{rema} If instead we rescale $e', f', h'$ by a factor of $1/\sqrt{a}$, the limit $\largeVir{k_0}{\infty} = \lim_{a \rightarrow \infty} \largeVir{k_0}{a}$ is still well-defined, but the fields $e', f', h'$ generate a nondegenerate rank $3$ Heisenberg vertex algebra and no longer commute with $\mathcal{V}^k$. In particular, the zero modes of $e', f', h'$ integrate to an $SU(2)$-action on $\mathcal{V}^k$ which will be needed later. 
\end{rema}

\subsection{The small $N=4$ superconformal algebra at central charge $-9$}
In this section we discuss two realizations of the small $N=4$ superconformal algebra at level $k = 1/2$, which has central charge $-9$. The first one is due to Drazen Adamovi\'c \cite{Ad1}, and the second one appeared implicitly in \cite{C}. Note that the affine \VOA{} of $\mathfrak{sl}_2$ then has level $-3/2$ and embeds conformally in the simple small $N=4$ superconformal algebra at central charge $-9$, which we call $Y$ \cite{AKFPP}. We will derive the branching rules.
\begin{theo}\label{thm:smallN4}
As a module for $SU(2) \times V_{-3/2}(\mathfrak{sl}_2)$ the simple small $N=4$ superconformal algebra at central charge $-9$ decomposes as 
\[
Y \cong \bigoplus_{m=0}^\infty \ \rho_{m\omega} \otimes V_{-3/2}(m\omega)
\]
where $V_{-3/2}(m\omega)$ denotes the Weyl module of $V_{-3/2}(\mathfrak{sl}_2)$ of highest-weight $m\omega$, $\rho_{m\omega}$ is the corresponding irreducible highest-weight representation of $SU(2)$ and $\omega$ is the fundamental weight of $\mathfrak{sl}_2$.
\end{theo}
\begin{proof}
Let us recall the construction of Adamovi\'c.
Consider the Wakimoto free field realization of the affine \VOA{} of $\mathfrak{sl}_2$ at level $-3/2$. He then realizes the simple small $N=4$ superconformal algebra as a kernel of screening charges inside an extension of the free field vertex algebra. Let us formulate this problem in terms of a $\beta\gamma$ \VOA{} together with a pair of fermionic $bc$-ghosts. We have the standard operator products
\[
\beta(z)\gamma(w) \sim (z-w)^{-1} \sim b(z)c(w).
\]
The affine vertex operator subalgebra is then generated by
\begin{equation}\nonumber
\begin{split}
e(z) &= \beta(z),\\
h(z) &=-2:\gamma(z)\beta(z): +:b(z)c(z):, \\ 
f(z) &=-:\gamma(z)\gamma(z)\beta(z): -\frac{3}{2}\partial \gamma(z) +:\gamma(z) b(z)c(z):.
 \end{split}
\end{equation}
We note that the conformal weight of $b$ is $3/2$ with respect to the Sugawara vector and the one of $c$ is $-1/2$. 
The four odd dimension $3/2$ fields are 
\begin{equation}\nonumber
\begin{split}
G^+(z) &= b(z),\\
G^-(z) &=-:\gamma(z)b(z): \\
\overline G^+(z)& = -2:\beta(z) \partial c(z): - :\partial\beta(z) c(z): \\
\overline G^-(z)&= - :b (\partial c) c:  - 2 :\beta \gamma \partial c:  - :(\partial \beta) \gamma c: - \frac{3}{2} \partial^2 c,
 \end{split}
\end{equation}

Let us call the \VOSA{} generated by these fields $Y$. Adamovi\'c then proves that $Y$ is simple \cite[Theorem 6.1]{Ad1} and that $Y$ coincides with the kernel of screenings on the $\beta\gamma bc$ \VOSA{} \cite[Corollary 6.2]{Ad1}. 
$Y$ is a module for the affine vertex operator subalgebra. The level of $\mathfrak{sl}_2$ is generic (see e.g. the computation of Section 3.1 of \cite{C}) and thus every Weyl module is simple. Let $\omega$ be the fundamental weight of $\mathfrak{sl}_2$ and we denote the Weyl module of weight $n\omega$ at level $-3/2$ by $V_{-3/2}(n\omega)$. 
Consider $X_n:= :b\partial b \dots \partial^{n-1} b:$ then $X_n$ clearly corresponds to an $\mathfrak{sl}_2$ highest-weight  vector of weight $n\omega$ and conformal weight $n^2/2+n$. Recall also that $SU(2)$ acts as outer automorphism group on $X$ and  since $b$ is a highest-weight vector for the standard representation $\rho_\omega$ the field $X_n$ must be one for $\rho_{n\omega}$. It follows that 
\[
\bigoplus_{n=0}^\infty \rho_{n\omega} \otimes V_{-3/2}(n\omega) \subset Y.
\]
We turn to the second construction of the small $N=4$ superconformal algebra at central charge $c=-9$. This construction immediately follows from \cite[Corollary 5.8]{C} together with \cite[Theorem 4.1]{CKLR}. Denote by $\mathcal H$ the rank one Heisenberg \VOA{} and by $\mathcal F_\mu$ the Fock module of highest-weight $\mu$ and conformal weight $\mu^2/2$. Let $\lambda$ be such that $\lambda^2=-1$. Then the case $p=2$ of \cite[Corollary 5.8]{C} says that there is a simple \VOA{} called $\mathcal W_2$ that satisfies
\[
\mathcal W_2 \cong \bigoplus_{m=0}^\infty V_{-3/2}(m\omega) \otimes \left( \mathcal F_{-m\lambda} \oplus \mathcal F_{(-m+2)\lambda} \oplus \dots \oplus  \mathcal F_{(m-2)\lambda} \oplus \mathcal F_{-m\lambda}\right)
 \]
as $V_{-3/2}(\mathfrak{sl}_2) \otimes \mathcal H$-module. The lattice \VOSA{}
\[
V_{\sqrt{-1}\mathbb Z}\bigoplus_{m\in \mathbb Z} \mathcal F_{m\lambda}
\]
extends $\mathcal H$ to a \VOSA{} and so by Theorem 4.1 of \cite{CKLR} (which follows from \cite[Theorem 3.1]{Li}, see also \cite{DL}) 
\[
\widetilde Y := \bigoplus_{m=0}^\infty\ (m+1) \ V_{-3/2}(m\omega)
\]
extends $V_{-3/2}(\mathfrak{sl}_2)$ to a larger \VOSA.  We thus see that $\widetilde Y\subset Y$ as $V_{-3/2}(\mathfrak{sl}_2)$-module. 
Moreover, the fields in $ V_{-3/2}(0) \oplus 2 V_{-3/2}(\omega)$ corresponding to the top level of the modules
generate a vertex operator subalgebra under operator product algebra. Since the conformal weight of the highest-weight states of $ V_{-3/2}(n\omega)$ is $n^2+n$ they must already close under operator products. They thus form a minimal W-superalgebra and its simple quotient is thus $Y$, i.e. the small $N=4$ superconformal algebra at $c=-9$ by the uniqueness result \cite[Theorem 3.1]{ACKL} (which generalized \cite[Lemma 8.2]{ACL}).

Let us summarize: we have shown firstly that 
$\widetilde Y\subset  Y$ as $V_{-3/2}(\mathfrak{sl}_2)$-module  and secondly $Y$ being a homomorphic image (namely the simple quotient) of a subalgebra of $\widetilde Y$. Both statements together can only be true if $\widetilde Y\cong Y$. 
\end{proof}

\subsection{The large $N=4$ superconformal algebra at central charge $-6$}

We now consider the simple algebra $\largeVir{1/2}{a}$, which has central charge $-6$. By \cite{AKFPP},  $V_{-(a+3)/2}(\mathfrak{sl}_2) \otimes V_{-(a^{-1}+3)/2}(\mathfrak{sl}_2)$ embeds conformally, that is, the Virasoro field $L$ is identified with the sum of the Sugawara vectors for the affine $\mathfrak{sl}_2$ subalgebras in $\largeVir{1/2}{a}$. Moreover, by Corollary \ref{cor:deformfam} in the case $k_0 = 1/2$, we see that $\largeVir{1/2}{a}$ is a deformable family of \VOSAs, that is, there exists a ring of the form $F_A$ where $A$ is at most countable, such that $F_A \otimes_F \largeVir{1/2}{a}$ is a free $F_A$-module. 

Let now $a$ be generic. The conformal weight of the top level subspace of the Weyl module $V_{-(a+3)/2}(m\omega)$ is $m(m+2)/(2-2a)$, so the conformal weight of  
$V_{-(a+3)/2}(m\omega)\otimes V_{-(a^{-1}+3)/2}(m'\omega)$ is a half-integer for generic $a$ if and only if $m=m'$ and in this case it equals $m(m+2)/2$. Since $\largeVir{1/2}{a}$ has finite-dimensional weight spaces and as such must be a direct sum of Weyl modules for generic $a$, by our conformal weight consideration the only possibility is
\[
\largeVir{1/2}{a} \cong \bigoplus_{m=0}^\infty N_m V_{-(a+3)/2}(m\omega) \otimes V_{-(a^{-1}+3)/2}(m\omega),
\]
with some multiplicities $N_m$. On the other hand, the graded character coincides with the one in the large $a$ limit. It follows from Theorem \ref{thm:smallN4} that we have $N_m=1$ for all $m$.
\begin{corol}\label{cor:larN=4}
As a $V_{-(a+3)/2}(\mathfrak{sl}_2) \otimes V_{-(a^{-1}+3)/2}(\mathfrak{sl}_2)$-module
\[
\largeVir{1/2}{a} \cong \bigoplus_{m=0}^\infty V_{-(a+3)/2}(m\omega) \otimes V_{-(a^{-1}+3)/2}(m\omega).
\]
\end{corol}
Let us write the character of $\largeVir{1/2}{a}$ for generic $a$ including Jacobi variables $y, z$, It is
\[
\text{ch}[\largeVir{1/2}{a}](y, z, q) = \sum_{m=0}^\infty \text{ch}[V_{-(a+3)/2}(m\omega)](q, y)\   \text{ch}[V_{-(a^{-1}+3)/2}(m\omega)](q, z)
\]
with
\[
\ch[V_{k}(m\omega)](z, q) = \frac{\left(z^{m+1}-z^{-(m+1)}\right)q^{\frac{m(m+2)-12k}{k+2}+\frac{1}{8}} }{\Pi(z)}
\]
with Weyl denominator 
\[
\Pi(z) = q^{\frac{1}{8}}\left(z-z^{-1}\right) \prod_{n=1}^\infty (1-z^2q^n)(1-q^n)(1-z^{-2}q^n).  
\]
 The Jacobi theta function of the lattice $\mathbb Z$ is 
\[
\theta_{\mathbb Z}(z, q) = \sum_{m\in\ZZ} q^{\frac{m^2}{2}}z^m.
\]
Let $k_1=-(a+3)/2$ and $k_2=-(a^{-1}+3)/2$ and note that 
\[
\frac{1}{k_1+2} + \frac{1}{k_2+2} = \frac{2}{1-a} +  \frac{2}{1-a^{-1}} =2.
\]
We now compute
\begin{equation}\nonumber
\begin{split}
\ch[\largeVir{1/2}{a}](z, w, q) &= \sum_{m=0}^\infty  \ch[V_{k_1}(m)](q, z)\ch[V_{k_2}(m)](q, w)\\
&=  \sum\limits_{m=0}^\infty\frac{ q^{\frac{(m+1)^2}{2}}\left((zw)^{m+1}+(zw)^{-(m+1)} -(zw^{-1})^{m+1}-(z^{-1}w)^{m+1}\right) }{\Pi(z) \Pi(w)} \\
&= \frac{\theta_{\mathbb Z}(zw, q) - \theta_{\mathbb Z}(zw^{-1}, q) }{\Pi(z)\Pi(w)}.
\end{split}
\end{equation}

\begin{corol} In the case $k = 1/2$, the algebra $\mathcal{V}^{1/2}$ in Lemma \ref{lemma:ainftylimit} is simple. We therefore obtain an exact sequence of vertex algebras
$$ 0 \rightarrow \cI_{Z}  \rightarrow \largeVir{1/2}{\infty} \rightarrow \smallVir{1/2} \rightarrow 0.$$ Here $\cI_Z \subseteq \largeVir{1/2}{\infty}$ is the ideal generated by the rank $3$ commutative Heisenberg algebra with generators $e', f', h'$.
\end{corol}




\begin{proof} This is immediate from Lemma \ref{lemma:ainftylimit}, Theorem \ref{thm:smallN4}, and Corollary \ref{cor:larN=4}.
\end{proof}

\begin{corol}
The character of $\largeVir{1/2}{a}$ converges to the meromorphic Jacobi form
\[
\text{ch}[\largeVir{1/2}{a}](q, w, z) = \frac{\theta_{\mathbb Z}(zw, q) - \theta_{\mathbb Z}(zw^{-1}, q) }{\Pi(z)\Pi(w)}
\]
for $|z|, |w| < |q|^{\pm 1}$.
\end{corol}
We note that the limit $z, w \rightarrow 1$ is by L'H\^opital's rule
\[
\lim_{z, w \rightarrow 1}\text{ch}[\largeVir{1/2}{a}](q, w, z) = \frac{1}{2}  \frac{\theta_{\mathbb Z}''(q)}{\eta(q)^6}
\]
with $\eta(q)$ the usual Dedekind eta-function and $\theta''(q) := \frac{d^2}{dz^2} \theta(z, q)\vert_{z=1}$. 
In other words the specialized character is a holomorphic modular form. This is a property of quasi-lisse \VOAs{} \cite{AK}. We wonder:
\begin{ques}
Is 
$\largeVir{1/2}{a}$  a deformable family of quasi-lisse \VOSAs?
\end{ques}
We remark that the super character is obtained from the character by replacing $z$ by $-z$, 
\[
\text{sch}[\largeVir{1/2}{a}](q, w, z) = \text{ch}[\largeVir{1/2}{a}](q, w, -z). 
\]

\subsection{Quantum Hamiltonian Reductions}

We turn to quantum Hamiltonian reductions, see e.g. \cite{Ar} as a reference. We are interested in applying the quantum Hamiltonian reduction to the affine sub algebras. For this consider two $bc$-ghost systems $bc$ and $b'c'$ and define the fields
\[
d(z) := b(z)e(z) +b(z), \qquad d'(z) := b'(z)e'(z) +b'(z)
\]
and denote their zero-modes by $d_0$ and $d_0'$. 
Let 
\[
W^k(a):= H_{d_0}(V(k, a)\otimes bc), \qquad X^k(a):= H_{d'_0}(W^k(a)\otimes b'c')
\]
and let
\[
W(a):= H_{d_0}(\largeVir{1/2}{a}\otimes bc), \qquad X(a):= H_{d'_0}(W(a)\otimes b'c').
\]
Let us first compute the characters of $W(a)$ and $X(a)$ first. They are obtained by taking the character of $\largeVir{1/2}{a}$ times the supercharacter of the ghosts times $q^{k/4}$ and then taking the limit of the Jacobi variable(s) to $q^{-1/2}$. The supercharacter of a pair of ghosts is
\[
\text{sch}[bc](q, z) = q^{\frac{1}{12}} \prod_{n=1}^\infty (1-z^2q^n)(1-z^{-2}q^{n-1})
\]
\begin{equation}
\begin{split}
\ch[W(a)](q, w) &= \lim_{z\rightarrow q^{-1/2}} \text{ch}[\largeVir{1/2}{a}](q, w, z)\text{sch}[bc](q, z)q^{k_1/4} \\
&= q^{-\frac{c}{24}+\frac{1}{6}} \frac{\theta_{\mathbb Z}(q^{-1/2}w, q) - \theta_{\mathbb Z}(q^{-1/2}w^{-1}, q) }{\eta(q)\Pi(w)}
\end{split}
\end{equation}
with $c= 1+3a^{-1}$ the central charge of the affine \VOSA{} $V_{k_2}(\mathfrak{osp}(1|2))$ ($k_2=-(a+3)/2$).
This is exactly the character of $V_{k_2}(\mathfrak{osp}(1|2))$ for generic $k_2$. 
Similarly
\begin{equation}
\begin{split}
\ch[X(a)](q) &= \lim_{z, w\rightarrow q^{-1/2}} \text{ch}[\largeVir{1/2}{a}](q, w, z)\text{sch}[bc](q, z)\text{sch}[bc](q, w)q^{k_1/4+k_2/4} \\
&= q^{-\frac{c}{24}+\frac{1}{12}} \frac{ \theta_{\mathbb Z}(q^{-1}, q) - \theta_{\mathbb Z}(1, q)}{\eta(q)^2}
\end{split}
\end{equation}
with $c=1+3a -6k_2 -2= 8=3(a+a^{-1})$ and this is exactly the character of the $N=1$ superconformal algebra at central charge $3/2 + 3(a+2+a^{-1})$ times a free fermion \VOSA. Note that this conclusion is obvious, since the character of the $N=1$ superconformal algebra can be obtained from the one of 
$V_{k}(\mathfrak{osp}(1|2))$ via Euler-Poincar\'e principle and the ghosts for this type of quantum Hamiltonian reduction do not only contain the $bc$-ghosts but also a pair of $\beta\gamma$ bosonic ghosts and a free fermion so that one sees that this Euler-Poincar\'e character differs from the one of $X(a)$ by a free fermion character. 
\begin{theo}\label{thm:red}
$W(a)$ is isomorphic to $V_{k_2}(\mathfrak{osp}(1|2))$ ($k_2=-(a+3)/2$)  and $X(a)$ is isomorphic to the $N=1$ superconformal algebra at central charge $3/2 + 3(a+2+a^{-1})$ times a free fermion \VOSA. 
\end{theo}
The Theorem follows directly from above character computation together with the next Lemma. But firstly, the $V_{k}(\mathfrak{osp}(1|2))$ is strongly generated by even fields $e', h', f'$ and odd fields $x', y'$ with OPEs
\begin{equation}\nonumber
\begin{split}
e'(z)f'(w) &\sim k(z-w)^{-2} + h'(w)(z-w)^{-1}, \qquad h'(z)h'(w) \sim 2k(z-w)^{-2}, \\
h'(z)e'(w) &\sim  2e'(w)(z-w)^{-1}, \qquad h'(z)f'(w) \sim  -2f'(w)(z-w)^{-1},\\
h'(z)x'(w) &\sim  x'(w)(z-w)^{-1}, \qquad h'(z)y'(w) \sim  -y'(w)(z-w)^{-1},\\
e'(z) y'(w) &\sim x'(w)(z-w)^{-1},\qquad f'(z) x'(w) \sim y'(w)(z-w)^{-1},\\
x'(z)y'(w) &\sim k(z-w)^{-2} + \frac{h'(w)}{2}(z-w)^{-1}, \qquad x'(z)x'(w) \sim -e'(w)(z-w), \\  y'(z)y'(w) &\sim f'(w)(z-w).\\
\end{split}
\end{equation}

\begin{lemma}
Let $k, a$ be generic, then $W^k(a)$ contains $V_{-((a+1)k+1)}(\mathfrak{osp}(1|2))$ as subalgebra 
and $X^k(a)$ contains the $N=1$ superconformal algebra at $$c=\frac{3 (1 + 2 k + 2 a k) (-1 + 4 k + 4 a k)}{2 (-1 + 2 k + 2 a k)},$$ as well as a free fermion as subalgebra.
\end{lemma}
\begin{proof}
We have $d_0( G^{-+})=0=d_0( G^{++})$. Moreover we have $d_0(c) =e+1$, i.e. $-e$ is in the same homology class as the vacuum $1$. Let us denote the class of $\frac{a+1}{\sqrt{2a}}G^{++}$ by $x'$ and the one of $-\frac{a+1}{\sqrt{2a}}G^{-+}$ by $y'$, then using that $-e(z)$ is in the same class as the vacuum we get the following OPEs
\begin{equation}\nonumber
\begin{split}
h'(z)x'(w) &\sim  x'(w)(z-w)^{-1}, \qquad h'(z)y'(w) \sim  -y'(w)(z-w)^{-1},\\
e'(z) y'(w) &\sim x'(w)(z-w)^{-1},\qquad f'(z) x'(w) \sim y'(w)(z-w)^{-1},\\
x'(z)y'(w) &\sim -((a+1)k+1)(z-w)^{-2} + \frac{h'(w)}{2}(z-w)^{-1}, \qquad x'(z)x'(w) \sim -e'(w)(z-w), \\  y'(z)y'(w) &\sim f'(w)(z-w)\\
\end{split}
\end{equation}
but these together with the relations of $e', h', f'$ are exactly the OPE relations of  $V_{-((a+1)k+1)}(\mathfrak{osp}(1|2))$.


For the second statement, note that $d'_0(c') = e'+1$, so $-e'$ is in the same cohomology class as the vacuum.
Since $x'_{(0)} x' = -e'$, it follows that the class $[x']$ satisfies 
$$[x'](z) [x'] (w) \sim (z-w)^{-1},$$ so it generates a free fermion algebra. Next, consider the Sugawara vector 
$$L^{\mathfrak{osp}(1|2)} = \frac{1}{2 (3 + 2 k)} :h'h':   + \frac{1}{3 + 2 k} :e' f': + \frac{1}{3 + 2 k} :f' e':    - \frac{1}{3 + 2 k} :x' y': +\frac{1}{3 + 2 k} :y' x': $$ in $V_k(\mathfrak{osp}(1|2))$. We correct it by setting
 $$L' = L^{\mathfrak{osp}(1|2)} +\frac{1}{2} \partial h'    - :b' \partial  c': - \frac{1}{2} : (\partial x') x:.$$ It is straightforward to check that
 $$d'_0(L') = 0$$ and that 
 $$L'(z) x'(w) \sim \frac{1}{2} (x' + :e' x':)(w)(z-w)^{-2} + (\partial x' + :e' \partial x')(w)(z-w)^{-1}.$$ Therefore $L'$ represents a cohomology class $[L']$ that commutes with $[x']$. Also, we compute
 $$(L'_{(0)} L') - \partial L' = -\frac{1}{2} \bigg(:e' (\partial^2 x') x': + :(\partial^2 x')x':\bigg) -\frac{1}{2} :(\partial e')(\partial x') x': - \frac{1}{24}\bigg((\partial^3 e') e': + \partial^3 e'\bigg),$$
 $$ (L'_{(1)} L') - 2 L' =   - \bigg( :e' (\partial x') x': + :(\partial x') x': \bigg) + \frac{1}{8} \bigg( : (\partial e') \partial e': - :(\partial^2 e') e': - \partial^2 e'\bigg).$$
$$ (L'_{(2)} L')  = \frac{1}{4} \bigg( :(\partial e')e': + \partial e' \bigg),$$
 $$(L'_{(3)} L')  = -\frac{3 + 10 k + 6 k^2}{3 + 2 k} 1 + \frac{1}{2} e' +\frac{1}{4} :e' e':.$$
 
 Using $-[x'] = [1]$ repeatedly, this shows that the class $[L']$ generates a Virasoro algebra of central charge $$c = -\frac{3 (1 + 2 k) (5 + 4 k)}{2 (3 + 2 k)}$$ which commutes with $[x']$.
 
Next, define the field
$$\psi = \frac{\sqrt{-1}}{\sqrt{3+2k}} \bigg( :h' x': + 2 :e' y': - (1+2k) :e' \partial x':\bigg).$$
We calculate
\begin{equation} d'_0(\psi) = 0,\end{equation}
\begin{equation} \begin{split} \psi(z) x'(w) & \sim -\frac{\sqrt{-1}(1+2k)}{\sqrt{3+2k}} (e' + :e' e':)(w)(z-w)^{-2}  \\ & - \frac{\sqrt{-1}}{2\sqrt{3+2k}}  \bigg((2+4k) :(\partial e') e': + (5+4k) \partial e'\bigg)(w)(z-w)^{-1}.\end{split} \end{equation}
This shows that the class $[\psi]$ is well-defined and commutes with $[x']$.

Next, we compute
\begin{equation} \begin{split} & \psi_{(2)} \psi = \frac{3 + 6 k}{3 + 2 k}  e' -\frac{4 (1 + 2 k)^2}{3 + 2 k} :e' e': -\frac{2(1 + 2 k)^2}{3 + 2 k} :e' e' e':,\\
& \psi_{(1)} \psi = \frac{3 + 6 k}{6 + 6 k}  \partial e' -\frac{4(1 + 2 k)^2}{3 + 2 k} :(\partial e') e': -\frac{3(1 + 2 k)^2}{3 + 2 k} :(\partial e') e' e':,\\
& \psi_{(0)} \psi = 2L - d'_0(R),\end{split}\end{equation} where 
\begin{equation} \begin{split} R & = \frac{1}{3 + 2 k} :h' h' c':  + \frac{4}{3 + 2 k} :e' f' c:   +\frac{2(1 + k)}{3 + 2 k} : (\partial h') c':  - \frac{4}{3 + 2 k} :x' y' c':   - 2: b' (\partial c') c':  \\ & -   :h' \partial c':  
 - \frac{(3 + 4 k)^2}{2 (3 + 2 k)} \partial^2 c' - \frac{3 (1 + 2 k)}{2(3 + 2 k)} :(\partial e') \partial c': + \frac{(1 + 2 k) (7 + 8 k)}{2(3 + 2 k)}:( \partial^2 e') c': \\& + \frac{(1 + 2 k)^2}{3 + 2 k} :(\partial e') e'( \partial c'):  + \frac{(1 + 2 k)^2}{3 + 2 k} :e' e' \partial^2 c': \end{split}\end{equation} It follows that in cohomology, we have 
 $$[\psi](z) [\psi](w) \sim -\frac{(1 + 2 k) (5 + 4 k)}{3 + 2 k} (z-w)^{-3} + 2[L'](w)(z-w)^{-1}.$$ It is also not difficult to check that
$$[L'](z) [\psi](w) \sim \frac{3}{2} [\psi](w) (z-w)^{-2} + \partial [\psi](w)(z-w)^{-1}.$$ Therefore $[L']$ and $[\psi]$ generate a copy of the $N=1$ algebra with central charge $c = -\frac{3 (1 + 2 k) (5 + 4 k)}{2 (3 + 2 k)}$, which commutes with $[x']$.

\end{proof}

\bibliographystyle{alpha}

\end{document}